\documentclass[letterpaper, 10 pt, conference]{ieeeconf}  

\IEEEoverridecommandlockouts                              
\usepackage{amsmath} 
\usepackage{amssymb}  
\usepackage{amsfonts}
\usepackage{dsfont}
\makeatletter
\newtheorem{lemma}{Lemma}[section]

\newtheorem{proposition}[lemma]{Proposition}
\newtheorem{remark}[lemma]{Remark}

\newcommand{\mR}{\mathbb{R}}

\newcommand{\X}{\mathcal{X}}
\newcommand{\J}{\mathcal{J}}
\newcommand{\Z}{\mathcal{Z}}

\newcommand{\R}{\mathcal{R}}

\newcommand{\U}{\mathcal{U}}

\newcommand{\cL}{\mathcal{L}}
\newcommand{\cH}{\mathcal{H}}

\newcommand{\bq}{\begin{equation}}
\newcommand{\eq}{\end{equation}}
\newcommand{\bma}{\begin{bmatrix}}
\newcommand{\ema}{\end{bmatrix}}
\DeclareMathOperator{\im}{im} 
\DeclareMathOperator{\rank}{rank}

\title{\LARGE \bf A Lagrange subspace approach to dissipation inequalities
}

\author{Arjan van der Schaft, Volker Mehrmann
\thanks{A.J. van der Schaft is with the Bernoulli Institute for Mathematics, Computer Science and AI, and Jan C. Willems Center for Systems and Control, University of Groningen, the Netherlands. E--mail: a.j.van.der.schaft@rug.nl, V. Mehrmann is with the Institut f\"ur Mathematik, TU-Berlin, Germany. E--mail: mehrmann@math.tu-berlin.de.}
}
\begin{document}

\maketitle
\thispagestyle{empty}
\pagestyle{empty}

\begin{abstract} The standard dissipation inequality for passivity is extended from storage functions to general Lagrange subspaces. This is shown to have some interesting consequences. A classical factorization result for passive systems is extended to this generalized case, making use of the newly defined concept of the Hamiltonian lift of a DAE system.
\end{abstract}

\section{Introduction}
Lyapunov and dissipation inequalities are at the heart of systems and control theory. Within the context of Riccati equations it is well-known that the symmetric solutions of them can be obtained by computing invariant Lagrange subspaces of a corresponding Hamiltonian matrix. Apart from the standard Lyapunov inequality $A^\top Q + QA \leq 0$ in the symmetric matrix $Q$, also the dual Lyapunov inequality $AX + XA^\top \leq 0$ arises at many places in the theory; notably in controllability and (stochastic) filtering. This already suggests the consideration of Lagrange subspaces as a means for their unification.
In this paper we take a close look at Lyapunov and dissipation inequalities by starting from a {\it general Lagrange subspace} point of view, and exploring its consequences. At least for dissipation inequalities this seems to be relatively new. After treating some generalities with regard to Lagrange subspaces and their invariance (Section II, see also the Appendix), we show in Section III how the standard dissipation inequality for passivity can be naturally extended from gradient vectors of {\it storage functions} to general {\it Lagrange subspaces}. Importantly, we note that this implies an {\it ordinary} dissipation inequality (with storage function determined by the Lagrange subspace) for an associated differential-algebraic equation (DAE) system. Then, in Section IV, we analyze the consequences of the satisfaction of this generalized dissipation inequality for the structure of the system. In particular we show that whenever the Lagrange subspace does {\it not} correspond to a storage function defined on the whole state space this necessarily implies loss of controllability. The rest of the paper is devoted to the {\it factorization} of systems satisfying the generalized dissipation inequality. First in Section V it is shown, via direct transfer matrix computations, how the satisfaction of the generalized dissipation inequality leads to the same factorization result as in the 'classical' dissipation case, but now with respect to the (generalized) transfer matrix of the DAE system alluded to above. In Section VI this result is extended to the state space setting, thereby also covering non-minimality. This extension is based on the newly defined concept of the Hamiltonian lift of a DAE system. Finally Section VII contains conclusions and outlook for further work.

\section{Lagrange subspaces and generalized Lyapunov inequalities}
Consider an $n$-dimensional linear state space $\X$, with dual space $\X^*$. A subspace $\cL \subset \X \times \X^*$ is called a {\it Lagrange subspace} if the canonical symplectic form $\J$ on $\X \times \X^*$, in matrix representation given as
\bq
\label{symplectic}
\J = \bma 0 & - I_n \\ I_n & 0 \ema,
\eq
is zero restricted to $\cL$, and furthermore $\cL$ is {\it maximal} with respect to this property. Equivalently, $\cL = \cL^{\perp}$, where ${}^{\perp}$ denotes the orthogonal companion with respect to $\J$. It follows that $\dim \cL=n$ for any Lagrange subspace.

It is easily seen (see Proposition \ref{A} in the Appendix) that any Lagrange subspace $\cL$ can be represented as
\bq
\label{image}
\cL=\im \bma P \\ S \ema
\eq
for certain square matrices $P,S$ satisfying
\bq
\label{sym}
S^\top P = P^\top S, \quad \rank \bma P \\ S \ema = n.
\eq
Furthermore, any subspace $\cL$ as in \eqref{image} with $P$ and $S$ satisfying \eqref{sym} is a Lagrange subspace. Denoting elements of $\X$ by $x$, and of its dual space $\X^*$ by $p$, it follows that all elements $(x,p) \in \cL$ can be expressed as
\bq
\label{image1}
x=Pz, \; p=Sz, \quad z \in \Z,
\eq
where $\Z$ is an $n$-dimensional parametrization space. The 'canonical' choice of $\Z$ is $\cL$, considered as a linear space in its own right. In this case, $P$ is given as the projection of $\cL \subset \X \times \X^*$ to $\X$, and $S$ is the projection of $\cL \subset \X \times \X^*$ to $\X^*$.

We note that any Lagrange subspace is endowed with an intrinsic bilinear form, which is given as
\bq
\langle (x_1,p_1), (x_2,p_2) \rangle = p_1^\top x_2=p_2^\top x_1, \ (x_i,p_i) \in \cL, i=1,2.
\eq
For $\cL$ given by \eqref{image}, \eqref{sym}, \eqref{image1} this bilinear form is simply given as
\bq
\langle (x_1,p_1), (x_2,p_2) \rangle = z_1^\top S^\top Pz_2
\eq
Obvious examples of a Lagrange subspace are \eqref{image} with $P=I,S=Q$, where $Q=Q^\top$ is a symmetric matrix, or $P=X,S=I$, where $X=X^\top$ is a symmetric matrix. However, there are many more cases 'in between' these two subclasses; see also Proposition \ref{B}.

Given any Hamiltonian function $H: \X \times \X^* \to \mR$ one defines the standard symplectic Hamiltonian dynamics on $\X \times \X^*$ as
\bq
\label{sympl}
\bma 0 & -I \\[2mm] I & 0 \ema \bma \dot{x} \\[2mm] \dot{p} \ema = \bma \frac{\partial H}{\partial x} \\[2mm] \frac{\partial H}{\partial p} \ema
\eq
A Lagrange subspace $\cL$ is {\it invariant} for \eqref{sympl} if and only if (see e.g. \cite{abraham})
\bq
H(x,p) = c, \mbox{ for all } (x,p) \in \cL
\eq
for some constant $c$. In the case of a quadratic Hamiltonian (as in all of this paper) the constant $c$ is actually zero.

With respect to the Hamiltonian $H(x,p)=p^\top Ax$, with $A$ some real $n \times n$ matrix, the Hamiltonian dynamics \eqref{sympl} is given as
\bq
\bma \dot{x} \\[2mm] \dot{p} \ema = \bma Ax \\[2mm] - A^\top p \ema,
\eq
and $\cL$ is invariant for this Hamiltonian dynamics if and only if $p^\top Ax=0$ for all $(x,p) \in\cL$. For any Lagrange subspace as given by \eqref{image}, \eqref{sym}, \eqref{image1} this holds if and only if
\bq
z^\top S^\top APz =0
\eq
for all $z$, or equivalently
\bq
\label{lyapgen}
S^\top A P + P^\top A^\top S =0, \, S^\top P = P^\top S, \, \rank \bma P \\ S \ema = n.
\eq
This overarches the two traditional {\it Lyapunov equations}
\bq
\label{lyap}
A^\top Q + QA=0, \, Q=Q^\top, \quad AX + XA^\top=0, \, X=X^\top,
\eq
corresponding to $P=I,S=Q$, respectively $P=X,S=I$.

Equation \eqref{lyapgen} therefore will be called the {\it generalized Lyapunov equation}; see \cite{MehT88} for related developments. Of course, in the case that $Q$ and $X$ are {\it invertible}, the two Lyapunov equations \eqref{lyap} are equivalent in the sense that by pre- and post-multiplication with the inverse of $Q$ or $X$ they can be transformed into each other. However, they are {\it not} equivalent if this is not the case.

Furthermore, the {\it generalized Lyapunov inequality} corresponds to $H(x,p) \leq 0$ for all $(x,p) \in \cL$. This means $z^\top S^\top APz \leq 0$ for all $z$, or equivalently
\bq
\label{lyapgen1}
S^\top A P + P^\top A^\top S \leq 0.
\eq
Classical examples are
\bq
\label{obs}
A^\top Q + QA= -C^\top C
\eq
with $Q$ the observability Grammian of the system $\dot{x}=Ax, y=Cx$, and
\bq
\label{cont}
A X + X A^\top = - BB^\top
\eq
with $X$ the controllability Grammian of the system $\dot{x}=Ax + Bu$. (The same equation appears in the covariance equation for linear stochastic systems.) The intrinsic difference between the two Lyapunov inequalities $\eqref{obs}$ and $\eqref{cont}$ is especially clear in case the system is {\it unobservable}, corresponding to a singular $Q$, or {\it uncontrollable}, corresponding to a singular $X$.

\section{The generalized dissipation inequality}
The ideas from the last section can be used for generalizing the {\it dissipation inequality} corresponding to passivity. In the case of a standard input-state-output system
\bq
\label{sys}
\begin{array}{rcl}
\dot{x} & = & Ax + Bu \\[2mm]
y & = & Cx + Du
\end{array}
\eq
this takes the following form. Consider the Hamiltonian
\bq
\label{ham}
H(x,p,u):= p^\top (Ax + Bu) - u^\top (Cx+Du).
\eq
Then the traditional dissipation inequality corresponding to passivity \cite{willems1972} is obtained by substituting $p=Qx, Q=Q^\top$, in $H(x,p,u) \leq 0$, so as to obtain
\bq
x^\top Q(Ax + Bu) - u^\top (Cx + Du) \leq 0, \ \mbox{for all } x,u.
\eq
This leads to the well-known linear matrix inequality, see e.g. \cite{willems1972},
\bq
\label{kyp}
\bma A^\top Q + QA & Q^\top B - C^\top \\[2mm] B^\top Q - C & - D -D^\top \ema \leq 0.
\eq
This inequality and its ramifications extends the famous Kalman-Yakubovich-Popov lemma; see e.g. \cite{willems1972} and the references quoted therein.

By replacing the storage function $\frac{1}{2}x^\top Q x$ (leading to the gradient vector $p=Qx$) by a Lagrangian subspace $\cL=\im \bma P \\ S \ema$ this extends to the {\it generalized dissipation inequality}
\bq
\label{kypgen}
\bma S^\top AP + P^\top A^\top S & S^\top B - P^\top C^\top \\[2mm] B^\top S - CP & - D -D^\top \ema \leq 0,
\eq
obtained by substituting $x=Pz$, $p=Sz$ in $H(x,p,u) \leq 0$.
If $P$ is singular, then \eqref{kypgen} can{\it not} be rewritten as a traditional dissipation inequality \eqref{kyp} corresponding to a storage function expressed in the state $x$. On the other hand, as we have seen above, there {\it is} a bilinear form, and thus a quadratic function, associated to the Lagrange subspace given by $P$, $S$, namely $V(z):= \frac{1}{2}z^\top S^\top Pz$. Then, using $P\dot{z} =\dot{x} = APz + Bu$, the generalized dissipation inequality \eqref{kypgen} is seen to be equivalent to
\bq
\frac{d}{dt} V(z) = z^\top S^\top P \dot{z} \leq u^\top (CPz + Du).
\eq
Thus, although \eqref{kypgen} can{\it not} be interpreted as a dissipation inequality for the original system \eqref{sys} in case $P$ is singular, it still {\it can} be interpreted as a dissipation inequality for the {\it DAE system} in the state variables $z$
\bq
\label{DAE}
\begin{array}{rcl}
P\dot{z} & = & APz + Bu, \quad z \in \Z, \\[2mm]
y & = & CPz + Du,
\end{array}
\eq
with singular pencil $sP-AP$, and  storage function $V(z)=\frac{1}{2}z^\top S^\top Pz$. Note that the equation space for the DAE system \eqref{DAE}, i.e., the co-domain of the mappings $P$ and $AP$, is given by $\X$.
\begin{remark}
Throughout this paper we do not impose any nonnegativity assumptions on $Q$ in \eqref{kyp}, or on $P$, $S$ in \eqref{kypgen}. Thus, strictly speaking we are dealing with {\it cyclo}-passivity instead of passivity; cf. \cite{passivitybook}. In particular, we do not assume nonnegativity of $S^\top P$. Geometrically the nonnegativity condition $S^\top P \geq 0$ can be formalized by requiring that the {\it symmetric} canonical form on $\X \times \X^*$, in matrix representation given as
\bq
\bma 0 & I_n \\ I_n & 0 \ema,
\eq
is nonnegative on $\cL =\im \bma P \\ S \ema$.
\end{remark}
\begin{remark}
The consideration of the generalized dissipation inequality \eqref{kypgen} suggests as a special case the following "dual" dissipation inequality (resulting from taking $S=I,P=P^\top =:X$ in \eqref{kypgen})
\bq
\label{kypgen2}
\bma AX + X A^\top  & B - X C^\top \\[2mm] B^\top - CX & - D -D^\top \ema \leq 0.
\eq
\end{remark}

\medskip

\section{Coordinate expressions}
Consider a general Lagrange subspace $\cL= \im \bma P \\ S \ema$. By allowing for coordinate transformations on $\X$ and $\Z$ we can always transform $P$ to the form
\bq
P = \bma I_k & 0 \\ 0 & 0 \ema,
\eq
where $k \leq n$ is the rank of $P$. Using $S^\top P=P^\top S$ it follows that $S$ needs to be of the corresponding form
\bq
S = \bma S_{11} & 0 \\ 0 & S_{22} \ema, \quad S_{11}=S^\top_{11}, S_{22}=S^\top_{22}.
\eq
Furthermore, since $\rank  \bma P \\ S \ema =n$, $S_{22}$ is invertible. In such coordinates, with partitioning $A,B,C$ accordingly, the generalized Lyapunov inequality \eqref{lyapgen1} can be seen to take the form
\bq
\bma A_{11}^\top S_{11} + S_{11}A_{11} & A^\top_{21}S_{22} & S_{11}B_1 - C_1^\top \\[2mm]
S_{22}A_{21} & 0 & S_{22}B_2, \\[2mm]
B_1^\top S_{11} - C_1 & B_2^\top S_{22} & -D -D^\top \ema \leq 0.
\eq
Because of the $0$-block this implies that $S_{22}A_{21}=0$, $S_{22}B_2=0$, and thus by invertibility of $S_{22}$, $A_{21}=0$ and $B_2=0$. Hence the system \eqref{sys} takes the block-triangular form
\bq
\label{sys1}
\begin{array}{rcl}
\dot{x}_1 & = & A_{11}x_1 + A_{12}x_2 + B_1u, \\[2mm]
\dot{x}_2 & = & A_{22}x_2, \\[2mm]
y & = & C_1x_1 + C_2x_2 + Du,
\end{array}
\eq
satisfying the {\it reduced} dissipation inequality
\bq
\bma A_{11}^\top S_{11} + S_{11}A_{11} & S_{11}B_1 - C_1^\top \\[2mm]
B_1^\top S_{11} - C_1 &  -D -D^\top \ema \leq 0.
\eq
In particular, it follows that for any initial condition $x(0)$ with $x_2(0)=0$, the system \eqref{sys1} leads to the reduced system
\bq
\label{sys2}
\begin{array}{rcl}
\dot{x}_1 & = & A_{11}x_1 + B_1u, \\[2mm]
y & = & C_1x_1 + Du,
\end{array}
\eq
that is satisfying the dissipation inequality with storage function $\frac{1}{2}x_1^\top S_{11}x_1$. Furthermore, it follows that any system \eqref{sys} satisfying the generalized dissipation inequality \eqref{kypgen} with $P$ singular is necessarily {\it uncontrollable} (since $x_2$ is independent of $u$ and $x_1$). In fact, it follows from \eqref{sys1} that the {\it controllable part} of the system \eqref{ABCD} is contained in $\im P$.
\begin{remark}
This is in line with the satisfaction of the dissipation inequality for the DAE system \eqref{DAE}, with storage function $\frac{1}{2}z^\top S^\top Pz$. In fact, using $A_{21}=0$, $B_2=0$ it follows that the DAE system \eqref{DAE} reduces to the ODE system
\bq
\begin{array}{rcl}
\dot{z}_1 & = & A_{11}z_1 + B_1u, \quad z=\bma z_1 \\ z_2 \ema \in \Z, \\[2mm]
y & = & C_1z_1 + Du,
\end{array}
\eq
solely in $z_1=x_1$; i.e., the part of \eqref{sys1} corresponding to any initial condition $x(0)$ with $x_2(0)= 0$.
\end{remark}
\begin{remark} When applied to the system $\dot{x}=Ax$ {\it without} inputs and outputs, \eqref{sys1} specializes to the triangular form
\bq
\label{sys12}
\begin{array}{rcl}
\dot{x}_1 & = & A_{11}x_1 + A_{12}x_2 \\[2mm]
\dot{x}_2 & = & A_{22}x_2 
\end{array}
\eq
satisfying the reduced Lyapunov inequality
\bq
A_{11}^\top S_{11} + S_{11}A_{11}  \leq 0
\eq
\end{remark}

\medskip 

Without using general coordinate transformations on $\X$ and $\Z$ one can still use the following general representation of Lagrange subspaces.
\begin{proposition}
\label{B}
Consider a Lagrange subspace $\cL \subset \X \times X^*$ given as $\cL = \im \bma P \\ S \ema$ for $n \times n$ matrices $P,S$ satisfying \eqref{sym}.
Suppose $\rank P=m \leq n=\dim \X$.
Then, possibly after permutation of the elements of $x$ and correspondingly $p$, there exists a splitting $x=\bma x_1 \\ x_2 \ema, p=\bma p_1 \\ p_2 \ema$ with $x_1,p_1$ both $m$-dimensional, and $x_2,p_2$ both $(n-m)$-dimensional, such that $\cL$ is represented as
\bq
\cL=\{(x,p) \in \X \times \X^* \mid \begin{bmatrix} p_1 \\ -x_2 \end{bmatrix} = W \begin{bmatrix} x_1 \\ p_2 \end{bmatrix} \}
\eq
with $W=W^\top$.
\end{proposition}
This proposition is well-known in symplectic geometry \cite{arnold, weinstein}, and in fact directly extends to Lagrange submanifolds. The quadratic function (in the subvectors $x_1,p_2$) defined by the symmetric matrix $W$ is called a {\it generating function} of the Lagrange subspace. A direct linear-algebraic proof of Proposition \ref{B} can be found in \cite{phDAE}.

As a consequence, by defining the following parametrization vector $z$, and partitioning $W$ accordingly,
\bq
z= \bma x_1 \\[2mm] p_2 \ema, \quad W= \bma W_{11} & W_{12} \\[2mm] W_{21} & W_{22} \ema,
\eq
we obtain
\bq
P= \bma I & 0 \\[2mm] -W_{21} & -W_{22} \ema, \quad S= \bma W_{11} & W_{12} \\[2mm] 0 & I \ema.
\eq
In particular, we see that
\bq
S^\top P = \bma W_{11} & 0 \\[2mm] 0 & -W_{22} \ema.
\eq

\section{Transfer matrix factorization}
Consider the DAE system \eqref{DAE}, with $\cL = \im \bma P \\ S \ema$ satisfying the generalized dissipation inequality \eqref{kypgen}. It follows that there exist matrices $M,N$ such that
\bq
\label{kypgen1}
\bma S^\top AP + P^\top A^\top S \! & \! S^\top B - P^\top C^\top \\[2mm] B^\top S - CP \! & \! - D -D^\top \ema = -
\bma M^\top \\[2mm] N^\top \ema \! \bma M \! & \! N \ema
\eq
(See \cite{reis2011lur} for related developments.)

We have the following result, directly generalizing a corresponding result for the classical dissipation inequality \eqref{kyp}; cf. \cite{anderson, willems1972}. Since $Ps-AP$ is singular, we cannot define the transfer matrix of the DAE system \eqref{DAE} in the standard way using the inverse of $Ps-AP$. However in the present case the transfer matrix of \eqref{DAE} still {\it can} be uniquely defined in a generalized form as
\bq
\label{Gs}
G(s):= CP(Ps - AP)^-B +D,
\eq
where for a matrix function $T$, $T^-$ denotes an appropriate generalized inverse. This is justified because $1)$ (as we have seen in the previous section) the generalized dissipation inequality implies that $\im B$ is contained in $\im P$ (see \eqref{sys1}, $2)$ the generalized inverse of $Ps-AP$ is premultiplied by $P$. Thus the expression $P(Ps - AP)^-B$ is indeed well defined. We obtain the following result.
\begin{proposition}\label{facto}
Consider the generalized transfer matrix of the DAE system \eqref{DAE} given by \eqref{Gs}.
Then \eqref{kypgen1} corresponds to the following factorization of transfer matrices
\bq
G(s) + G^\top (-s) = K^\top (-s) K(s)
\eq
where
\bq
\label{Ks}
K(s):= M(Ps - AP)^{-}B +N
\eq
(Note that $\ker P \subset \ker M$ so again $K(s)$ is well defined.)
\end{proposition}
\noindent
\begin{proof}
For simplicity assume $D=0$, in which case $N=0$ in \eqref{kypgen1}, and thus \eqref{kypgen1} amounts to
\bq
\label{STAP}
S^\top AP + P^\top A^\top S = - M^\top M, \quad CP=B^\top S
\eq
Then, using similar arguments as in \cite{anderson}, and additionally using $S^\top P = P^\top S$,
\bq
\begin{array}{l}
G(s) + G^\top (-s) =  \\[2mm]
CP(Ps-AP)^{-}B + B^\top (-P^\top s - P^\top A^\top)^{-}P^\top C^\top =\\[2mm]
 B^\top S(Ps-AP)^{-}B + B^\top (-P^\top s - P^\top A^\top)^{-}S^\top B = \\[2mm]
\left[B^\top S + B^\top (-P^\top s - P^\top A^\top)^{-}S^\top (Ps -AP) \right] \cdot  \\[2mm]
(Ps-AP)^{-}B =\\[2mm]
 B^\top (-P^\top s - P^\top A^\top)^{-} \cdot \\[2mm]
 \left[ (- P^\top s - P^\top A^\top )S - S^\top (Ps -AP) \right] \cdot \\[2mm]
 (Ps-AP)^{-}B =\\[2mm]
 B^\top (-P^\top s - P^\top A^\top)^{-} \cdot \\[2mm]
 \left[ -P^\top A^\top S - S^\top AP \right] \cdot (Ps-AP)^{-}B =\\[2mm]
 B^\top (-P^\top s - P^\top A^\top)^{-} M^\top M  (Ps-AP)^{-}B \\[2mm]
 = K^\top (-s) K(s).
\end{array}
\eq
The proof for the case $D\neq 0$ can be performed in a similar way (see \cite{anderson} for details in a similar context), or by first removing the feedthrough term by an extension as discussed in \cite{MehU22_ppt}.
\end{proof}

\section{The Hamiltonian lift of DAE systems and factorization}
In this section we will show how the factorization, performed in the previous section using direct transfer matrix computations, can be also obtained from a pure state space point of view (and so including the uncontrollable part of the system).
First recall, see \cite{crouch-vds}, that any ordinary input-state-output system
\bq
\label{ABCD}
\begin{array}{rcl}
\dot{x} & = &Ax + Bu, \\[2mm]
y & = & Cx +Du,
\end{array}
\eq
on $\X$ can be {\it lifted} to a Hamiltonian input-output system on $\X \times \X^*$. (See \cite{brockett, 1981, crouch-vds, passivitybook} for the definition of a Hamiltonian input-output system.) Indeed, one considers the canonical symplectic form $\J$ on $\X \times \X^*$ as in \eqref{symplectic}, together with the Hamiltonian as considered before
\bq
\label{ham1}
H(x,p,u)= p^\top (Ax + Bu) - u^\top (Cx + Du).
\eq
This leads to the {\it Hamiltonian input-output system}
\bq
\label{hamio}
\begin{array}{rcl}
\J \bma \dot{x} \\[2mm] \dot{p} \ema & = & \bma \frac{\partial H}{\partial x} \\[2mm] \frac{\partial H}{\partial p} \ema =
\bma   A^\top p - C^\top u \\[2mm]  Ax+Bu \ema, \\[7mm]
y_{\mathrm{ham}} & = & - \frac{\partial H}{\partial u} = -B^\top p + Cx + (D + D^\top)u,
\end{array}
\eq
also called the {\it Hamiltonian lift} of the original system \eqref{ABCD}.
(The definition immediately extends to the nonlinear case as well; see \cite{crouch-vds} for details.) It is immediately verified that the transfer matrix of \eqref{hamio} is given as
\bq
G'(s) + G'^\top (-s),
\eq
where $G'(s)= C(Is - A)^{-1}B + D$ is the transfer matrix of \eqref{ABCD}. This lifting is instrumental for various purposes; see e.g. \cite{passivitybook} for a detailed discussion.

The described lifting construction can be extended to DAE systems as follows. Consider first a {\it general} linear DAE system
\bq
\label{DAE1}
\begin{array}{rcl}
E \dot{x} & = & Ax + Bu, \\[2mm]
y & = & Cx + Du,
\end{array}
\eq
where $E$ and $A$ are mappings $E: \X \to \R$, $A: \X \to \R$, for some linear equation space $\R$, together with $B: \U \to \R$, with $\U=\mR^m$ the input space. Then define the skew-symmetric form $\J_E$ on $\X \times \R^*$ as
\bq
\J_E := \bma 0 & - E^\top \\[2mm] E & 0 \ema.
\eq
Since in general $E$ is not invertible, the skew-symmetric form $\J_E$ is degenerate, and defines a {\it pre-symplectic form} \cite{abraham, arnold}. Similarly to \eqref{ham}, define the Hamiltonian $H': \X \times \R^* \times \U \to \mR$ as
\bq
H'(x,v,u) := v^\top (Ax + Bu) - u^\top (Cx + Du).
\eq
Then define the {\it Hamiltonian lift} of \eqref{DAE1} as the input-output Hamiltonian DAE system
\bq
\begin{array}{rcl}
\J_E \bma \dot{x} \\[2mm] \dot{v} \ema & = & \bma \frac{\partial H'}{\partial x} \\[2mm] \frac{\partial H'}{\partial v} \ema =
\bma   A^\top v - C^\top u \\[2mm]  Ax+Bu \ema, \\[7mm]
y_{h} & = & - \frac{\partial H'}{\partial u} = -B^\top v + Cx + (D + D^\top)u.
\end{array}
\eq
Applied to the specific DAE system \eqref{DAE} with the corresponding Hamiltonian
\bq
\label{cH}
\cH(z,v,u) := v^\top (APz + Bu) - u^\top (CPz + Du)
\eq
this yields the following Hamiltonian lift of \eqref{DAE}
\bq
\label{hamlift1}
\begin{array}{rcl}
\bma 0 & -P^\top \\[2mm] P & 0 \ema \bma \dot{z} \\[2mm] \dot{v} \ema & = &
\bma   P^\top A^\top v - P^\top C^\top u \\[2mm]  APz+Bu \ema, \\[7mm]
\hat{y}_{h} & = &  -B^\top v + CPz + (D + D^\top)u,
\end{array}
\eq
living on $\Z \times \X^*$. See \cite{Meh91a} for a related Hamiltonian in the context of optimal control problems with DAE constraints.

The associated generalized transfer matrix of the Hamiltonian lift \eqref{hamlift1} is given as
\bq
\begin{array}{l}
\bma CP \! & \! -B^\top \ema \bma 0  \! & \! -sP^\top - P^\top A^\top \\[2mm] sP - AP \! & \! 0 \ema ^{-} \bma -P^\top C^\top \\[2mm] B \ema \\[7mm]
+ \; (D + D^\top) =  CP(sP-AP)^{-}B + D \\[2mm]
+ \; B^\top (-sP^\top - P^\top A^\top)^{-} P^\top C^\top +D^\top
\end{array}
\eq
Thus, similar to the Hamiltonian lift of the standard input-state-output system \eqref{ABCD}, the generalized transfer matrix of the Hamiltonian lift of \eqref{DAE} is equal to $G(s) + G^\top (-s)$, with $G(s)=CP(Ps - AP)^{-}B +D$ the generalized transfer matrix of \eqref{DAE}.
\begin{remark}
This holds as well for the Hamiltonian lift of a {\it general} DAE system \eqref{DAE1}.
\end{remark}
We will now show how the factorization result $G(s) + G^\top (-s) = K^\top (-s) K(s)$ as obtained in the previous section by generalized transfer matrix computations can be also obtained by state space methods applied to the Hamiltonian lift \eqref{hamlift1}. This provides extra insight, and, importantly, handles the intrinsic {\it non-minimality} (due to uncontrollability) arising from the generalized dissipation inequality in case $P$ is singular. Furthermore, the state space point of view is expected to facilitate the generalization to the nonlinear case; cf. \cite{passivitybook} for the standard nonlinear input-state-output case.

Thus let $\cL=\im \bma P \\ S \ema$ be a solution to \eqref{kypgen}. Then consider the {\it state space transformation}
\bq
\label{transfo}
z=\bar{z}, v=\bar{v} +Sz
\eq
for \eqref{hamlift1}. This transforms the Hamiltonian $\cH(z,v,u)$ defined in \eqref{cH} into
\bq
\begin{array}{l}
\bar{\cH}(\bar{z},\bar{v},u) := \bar{v}^\top (AP\bar{z} + Bu) - u^\top (CP\bar{z} + Du), \\[2mm]
 \quad +\, \bar{z}^\top S^\top AP \bar{z} + \bar{z}^\top S^\top Bu - u^\top CP \bar{z} - u^\top Du.
\end{array}
\eq
Making use of \eqref{kypgen1} this yields
\bq
\bar{\cH}(\bar{z},\bar{v},u) = \bar{v}^\top (AP\bar{z} + Bu) - \frac{1}{2}(M\bar{z} + Nu)^\top (M\bar{z} + Nu).
\eq
Furthermore, the state space transformation \eqref{transfo} leaves the pre-symplectic form $\bma 0 & -P^\top \\P & 0 \ema$ invariant, since
\bq
\begin{array}{l}
\bma I & S^\top \\[2mm] 0 & I \ema \bma 0 & -P^\top \\[2mm] P & 0 \ema \bma I & 0 \\[2mm] S & I \ema = \\[6mm]
\bma -P^\top S + S^\top P & -P^\top \\[2mm] P & 0 \ema = \bma 0 & -P^\top \\[2mm] P & 0 \ema,
\end{array}
\eq
in view of $S^\top P = P^\top S$.
Hence the state space transformation \eqref{transfo} transforms the Hamiltonian lift \eqref{hamlift1} into another input-output Hamiltonian representation of the {\it same} system, still defined with respect to the {\it same} presymplectic form, but now with Hamiltonian $\bar{\cH}$. This new representation (in the new coordinates $\bar{z},\bar{v}$) is given as
\bq
\label{hamlift2}
\begin{array}{rcl}
\bma 0 & -P^\top \\[2mm] P & 0 \ema \bma \dot{\bar{z}} \\[2mm] \dot{\bar{v}} \ema  & = &
\bma   P^\top A^\top \bar{v} - M^\top (M \bar{z} + Nu) \\[2mm]
AP \bar{z}+Bu \ema, \\[6mm]
\widehat{y}_{h}  & =  &- B^\top \bar{v}  + N^\top (M\bar{z} + Nu).
\end{array}
\eq
It can be directly verified that the generalized transfer matrix of \eqref{hamlift2} is given by $K^\top (-s) K(s)$, with $K(s)$ defined by \eqref{Ks}.

From a state space point of view \eqref{hamlift2} is most easily interpreted in case $D=0$ (see \cite{MehU22_ppt} for the removal of $D$ by an extension). In fact, $D=0$ results in \eqref{STAP}, and thus the Hamiltonian lift \eqref{hamlift2} simplifies to
\bq
\label{hamlift3}
\begin{array}{rcl}
\bma 0 & -P^\top \\[2mm] P & 0 \ema \bma \dot{\bar{z}} \\[2mm] \dot{\bar{v}} \ema  & = & \bma   P^\top A^\top \bar{v} - M^\top M \bar{z}  \\[2mm]
AP \bar{z}+Bu \ema, \\[6mm]
\widehat{y}_{h}  & = &  -B^\top \bar{v}.
\end{array}
\eq
This is seen to be the {\it series interconnection} of the system (with generalized transfer matrix $K(s)$)
\bq
\label{primal}
P \dot{\bar{z}} =  AP \bar{z} + B\bar{u}, \quad \bar{y} = M\bar{z},
\eq
with inputs $\bar{u}$ and outputs $\bar{y}$, and its {\it adjoint system} (with generalized transfer matrix $K^\top(-s)$)
\bq
\label{dual}
P^\top \dot{\bar{v}} = - P^\top A^\top \bar{z} + M^\top u_a, \quad y_a = -B^\top \bar{v},
\eq
with inputs $u_a$ and outputs $y_a$. Indeed, by substituting $u_a=\bar{y}$ in \eqref{primal}, \eqref{dual} one recovers \eqref{hamlift3} with $u=\bar{u}$ and $\widehat{y}_h=y_a$. See e.g. \cite{KuM} for the definition of the adjoint system of a DAE system.

\begin{remark}
Note that the original input-output Hamiltonian representation \eqref{hamlift1} (in the old coordinates $z,v$) is equal to the {\it parallel interconnection} (i.e., $\widehat{y}_h = y + y_a$, $u_a=u$) of the original DAE system \eqref{DAE} with its adjoint system
\bq
\label{dual2}
\begin{array}{rcl}
P^\top \dot{v} & = & - P^\top A^\top v + C^\top u_a \\[2mm]
y_a & = & - B^\top v + D^\top u_a.
\end{array}
\eq
Thus by the generalized dissipation inequality this {\it parallel} interconnection is written as a {\it series} interconnection.
\end{remark}
\begin{remark}
Note that \eqref{dual} together with \eqref{primal} implies
\bq
\frac{d}{dt} \bar{v}^\top P \bar{z} = \dot{\bar{v}}^\top P \bar{z} + \bar{v}^\top P \dot{\bar{z}} = u_a^\top \bar{y} - \bar{u}^\top y_a,
\eq
which can be regarded as the defining property of the adjoint system; cf. \cite{crouch-vds}.
Similarly for \eqref{dual2} with respect to \eqref{DAE}.
\end{remark}

\section{Conclusions and outlook}
It has been shown how Lyapunov and dissipation inequalities can be generalized from quadratic Lyapunov and storage functions to Lagrange subspaces. This gives rise to the consideration of an associated DAE system on a parametrization space. This DAE system satisfies the classical dissipation inequality with respect to a storage function that is determined by the Lagrange subspace. Next it has been shown how the same factorization result can be obtained as for the classical dissipation inequality. The state space treatment of this result, making use of the Hamiltonian lift of a DAE system, does not rely on minimality, and suggests generalization to the nonlinear case. Another venue for future research is the extension of the Lagrange subspace methodology from standard ODE systems (as in this paper) to DAE systems. As a final result, it has been noted how the generalized dissipation inequality implies lack of controllability in case the Lagrange subspace does not correspond to a storage function defined on the whole state space. This is another indication that the traditional emphasis in systems and control theory on {\it minimal} systems may not be so relevant any more; also in the light of electrical circuit and neural network theory.

\section{Appendix}

\begin{proposition}
\label{A}
A subspace $\cL \subset \X \times \X^*$ with $\dim \X=n$ is a Lagrange subspace if and only if there exist $n \times n$ matrices $P,S$ satisfying
\bq
\label{sp}
S^\top P = P^\top S , \; \rank \begin{bmatrix} S^\top & P^\top \end{bmatrix} =n
\eq
such that (see \eqref{image})
\bq
\label{image2}
\cL=\{(x,p) \in \X \times \X^* \mid \exists z \in \Z=\mR^n \mbox{ s.t. }\begin{bmatrix} x \\ p \end{bmatrix} = \begin{bmatrix} P \\ S \end{bmatrix}z \}.
\eq
\end{proposition}
\begin{proof}
The 'if' direction follows by checking that $\bma x_1^\top & p_1^\top \ema \bma 0 & -I \\ I & 0 \ema \bma x_2 \\ p_2 \ema =0$ for any two pairs $(x_i,p_i)$ with $x_i=Pz_i, p_i=Sz_i, i=1,2$, and $P$, $S$ satisfying \eqref{sp}.\\
For the 'only if' direction we note that any $n$-dimensional subspace $\cL$ can be written as in \eqref{image2} for certain $n \times n$ matrices $ P$, $S$ satisfying $\rank \begin{bmatrix} P^\top & S^\top \end{bmatrix} =n$. Then take any two pairs $(x_i,p_i) \in \cL$ with $x_i=Pz_i$, $p_i=Sz_i$, $i=1,2$. Since $\cL$ is a Lagrange subspace it follows that
\bq
\begin{array}{l}
0 =  \bma x_1^\top & p_1^\top \ema \bma 0 & -I \\ I & 0 \ema \bma x_2 \\ p_2 \ema = z_2^\top S^\top Pz_1 - z_1^\top S^\top Pz_2 \\[3mm]
\quad = -z_1^\top (S^\top P - P^\top S)z_2
 \end{array}
\eq
for all $z_1,z_2$, implying that $S^\top P = P^\top S$.
\end{proof}

\end{document}